\documentclass{amsart}
\usepackage{amssymb,latexsym,enumitem,graphics}
\pdfoutput=1

 \newtheorem{thm}{Theorem}[section]
 \newtheorem{cor}[thm]{Corollary}
 \newtheorem{lem}[thm]{Lemma}
 \newtheorem{prop}[thm]{Proposition}
 \theoremstyle{definition}
 
 \theoremstyle{remark}
 \newtheorem{rem}[thm]{Remark}
 
 \numberwithin{equation}{section}

\theoremstyle{remark}

\begin{document}
\title[Complex symmetric Weighted Composition--Differentiation Operators]{Complex symmetric Weighted Composition--Differentiation Operators of order $n$ on the Weighted Bergman Spaces}
\author[M. Moradi and M. Fatehi]{Mahbube Moradi and Mahsa Fatehi}

\date{July 10, 2020}
\address{Department of Mathematics\\ Shiraz Branch, Islamic Azad University\\
Shiraz, Iran}
\email{mmoradi8524@gmail.com}
\address{Department of Mathematics\\ Shiraz Branch, Islamic Azad University\\
Shiraz, Iran}
\email{fatehimahsa@yahoo.com}

\begin{abstract}
We study the complex symmetric structure of  weighted composition--differentiation operators of order $n $ on the weighted Bergman spaces $A_{\alpha}^2$ with respect to some conjugations. Then we provide some examples of these operators.
\end{abstract}
\maketitle
\thispagestyle{empty}
\section{Preliminaries}

Let $\mathbb{D}$ denote the open disk in the complex plane $\Bbb{C}$. For $\alpha>-1$, the {\it weighted Bergman space} $A_{\alpha}^2$ is the weighted Hardy space consisting of all analytic functions
$f(z)= \sum_{n=0}^{\infty} a_n z^n $ on $\mathbb{D}$ such that $\|f\|^2= \sum_{n=0}^{\infty} |a_n|^2 \beta(n)^2 < \infty,$
where for each nonnegative integer $n$,
$\beta(n)=\|z^n\|= \sqrt{\frac{n! \Gamma(\alpha+2)}{\Gamma(n+\alpha+2)}}$.
 The inner product of this space is given by $\langle \Sigma_{j=0}^{\infty} a_j z^j, \Sigma_{j=0}^{\infty} b_j z^j   \rangle = \Sigma_{j=0}^{\infty} a_j  \overline{b_j} \beta(j)^2$
for functions in $A_{\alpha}^2$.
It is well-known that this space is a reproducing kernel Hilbert space, with  kernel functions $K_w^{(m)}$ for any $w \in \mathbb{D}$
and nonnegative integer $m$ that $
\langle f, K_w^{(m)} \rangle = f^{(m)}(w)$
 for each $f \in A_{\alpha}^2$.
To simplify notation, we use $K_w$, when $m=0$.
 We recall that
$
K_w(z)= 1/(1- \overline{w}z)^{\alpha+2}=\sum_{j=0}^{\infty}\frac{\overline{w}^{j}z^{j}}{\beta(j)^{2}}
$
and for $m> 1$,
\begin{eqnarray*}
K_w^{(m)}(z)= \frac{(\alpha+2) \ldots (\alpha+m+1) z^m}{(1-\overline{w}z)^{m+\alpha+2}}= \frac{m! z^m}{\beta (m)^2 \, (1 - \overline{w}z)^{m+\alpha+2}}.
\end{eqnarray*}
Moreover, for each nonnegative integer $m$, we have
\begin{eqnarray*}
\|K_w^{(m)} \|^2 = \sum_{j=m}^{\infty} \frac{(|w|^2)^{j-m}}{\beta(j)^2} \bigg(\frac{j!}{(j-m)!}\bigg)^2.
\end{eqnarray*}\par
We recall that $H^{\infty}(\mathbb{D})=H^{\infty}$ is the
space of all bounded analytic functions defined on $\mathbb{D}$,
with supremum norm $\|f\|_{\infty}=\sup_{z\in \mathbb{D}}|f(z)|$.
Let $P_{\alpha}$ be the projection of $L^{2}({\mathbb{D}}, dA_{\alpha})$
onto $A^{2}_{\alpha}$.
Given a function $h \in L^{\infty}(\mathbb{D})$, the  {\it Toeplitz  operator}
 $T_{h}$ on $A^{2}_{\alpha}$ is defined by $T_{h}(f)=P_{\alpha}(hf)$. It is easy to see that if $h \in H^{\infty}$, then $T_{h}(f)=h\cdot f$.
 For $\varphi$ an analytic self-map of $\mathbb{D}$, let $C_{\varphi}$ be the {\it composition operator}
 so that $C_{\varphi}(f) = f \circ \varphi$ for any $f \in A^{2}_{\alpha}$. All composition operators and Toeplitz operators are bounded on $A^{2}_{\alpha}$. A natural generalization of a composition operator is an operator that takes $f$ to $\psi \cdot f \circ \varphi$, where $f \in A^{2}_{\alpha}$ and $\psi$ is an analytic map on $\mathbb{D}$. This operator is called a {\it weighted composition operator} and is  denoted by $C_{\psi,\varphi}$. \par
 For a positive integer $n$,  we define the {\it differential operator of order $n$}  on $A^{2}_{\alpha}$ by $D^{(n)}(f)=f^{(n)}$.
 We know that the differentiation operator of order $n$ is not bounded on $A^{2}_{\alpha}$, however, for many analytic self-maps $\varphi$, the operator $C_{\varphi}D^{(n)}$ is bounded on $A^{2}_{\alpha}$.
 The study of such operators  were initially addressed by Hibschweiler, Portnoy and Ohno in \cite{hp} and \cite{s3}
  and  afterwards has been noticed by many  researchers (see also \cite{s1}, \cite{s2} and \cite{s4}).
 Ohno \cite{s3}  and S. Stevi\' c \cite{s4} characterized boundedness and compactness of  $C_{\varphi}D^{(1)}$ on the Hardy space and $C_{\varphi}D^{(n)}$ on the weighted Bergman spaces, respectively.
    The bounded operator $C_{\varphi}D^{(n)}$  is denoted by $D_{\varphi,n}$ and called by {\it composition-differentiation operator of order $n$}. For an analytic function $\psi$ on $\mathbb{D}$, the {\it weighted composition-differentiation operator of order $n$}  on $A^{2}_{\alpha}$ is defined by $D_{\psi,\varphi,n}(f)=\psi\cdot(f^{(n)}\circ \varphi)$. Note that $D_{\psi,\varphi,n}$ is actually the product of the Toeplitz operator $T_{\psi}$ and $D_{\varphi,n}$, whenever $\psi \in H^{\infty}$ and $D_{\varphi,n}$ is bounded.   To avoid trivial situations, we will be assuming throughout this paper that $\psi$ is not identically $0$ and that $\varphi$ is nonconstant.\par
A bounded operator $T$ is called {\it complex symmetric operator} on a complex Hilbert space $H$ if there exits a conjugation $C$ (i.e. an antilinear, isometric involution) so that $CT^{\ast}C=T$ and we say that $T$ is $C$-symmetric. Complex symmetric operators has been considered initially  on Hilbert spaces of holomorphic functions by Garcia and  Putinar in \cite{gar1} and \cite{gar2}. Then complex symmetric weighted composition operators were considered in \cite{s0}, \cite{ham}, \cite{s5} and \cite{s55}.
In this paper, we use the symbol $J$ for the special conjugation that $(Jf)(z)=\overline{f(\overline{z})}$ for each analytic function $f$.\par
 For each $z \in \Bbb{C}$, we write $z = |z|e^{i\theta}$, where $0 \leq \theta < 2\pi$. The argument of $z$ is denoted by $\mbox{Arg}(z) = \theta$  and we set $\mbox{Arg}(0)=0$.
\section{complex symmetric operators $D_{\psi, \varphi, n}$}

  For $\varphi$ an analytic self-map of $\mathbb{D}$, the generalized Nevanlinna counting function $N_{\varphi, \alpha+2}$ is defined by  $N_{\varphi, \alpha+2} (w)=\sum_{\varphi(z)=w}[\ln(1/|z|)]^{\alpha+2}$, where $w$ belongs to $\mathbb{D}\setminus \{\varphi(0)\}$. The next proposition gives a necessary and sufficient condition for  $D_{\varphi, n}$ to be  bounded and compact.

\begin{prop} \label{f0}
\cite[Theorem 9]{s4}.
Let $\varphi$ be an analytic self-map of $\mathbb{D}$, $n \in \Bbb{N}$ and $\alpha >-1$. Then the following hold.
\begin{itemize}
\item[$a)$]
An operator
$D_{\varphi, n}:A_{\alpha}^2 \rightarrow  A_{\alpha}^2$ is bounded if and only if
\begin{eqnarray*}
N_{\varphi, \alpha+2} (w) = O  \left(  [ \ln (1/|w|)  ]^{\alpha+2+ 2n}  \right) \,\,\,\,\, (|w| \rightarrow 1).
\end{eqnarray*}
\item[$b)$]
An operator
$D_{\varphi,n}:A_{\alpha}^2 \rightarrow  A_{\alpha}^2$ is compact if and only if
\begin{eqnarray*}
N_{\varphi, \alpha+2} (w) = o  \left(  [ \ln (1/|w|)  ]^{\alpha+2+ 2n}  \right)   \,\,\,\,\, (|w| \rightarrow 1).
\end{eqnarray*}
\end{itemize}
\end{prop}

The next corollary follows from  Proposition \ref{f0} for the case that $\varphi$ is univalent on $\mathbb{D}$ (note that $\mbox{ln} (\frac{1}{|w|})$ is comparable to $1-|w|$ as $|w| \rightarrow 1^{-})$.

\begin{cor}  \label{f1}
Let $\varphi$ be a univalent self-map of $\mathbb{D}$. Then
\begin{itemize}
\item[$a)$]
An operator
$D_{\varphi, n}$ is bounded on $A_{\alpha}^2 $ if and only if
\begin{eqnarray*}
\sup_{w \in \mathbb{D}}  \frac{(1- |w|)^{\alpha+2}}{ (1-|\varphi (w)| )^{\alpha+2+2n}} <\infty.
\end{eqnarray*}
\item[$b)$]
An operator
$D_{\varphi,n}$ is compact on $A_{\alpha}^2 $ if and only if
\begin{eqnarray*}
\lim_{|w| \rightarrow 1}  \frac{(1- |w|)^{\alpha+2}}{ (1-|\varphi (w)|  )^{\alpha+2+2n}} =0.
\end{eqnarray*}
\end{itemize}
\end{cor}

Note that Corollary \ref{f1} shows that  $D_{\varphi,n}$ is bounded if $\varphi$ does note have a finite angular derivative at any points on $\partial \mathbb{D}$ (see \cite[Theorem 2.44]{S.0}) and so in this case $D_{\psi,\varphi,n}$ is bounded on $A_{\alpha}^2 $, when $\psi \in H^{\infty}$. We state the following lemma which will be used in this paper.
\begin{lem} \label{f2}
If  an operator $D_{\psi, \varphi, n}$ is bounded on   $A_{\alpha}^2 $, then
\begin{eqnarray*}
D_{\psi, \varphi, n}^{\ast} (K_w)= \overline{\psi(w)}  K_{\varphi (w)}^{(n)}.
\end{eqnarray*}
\end{lem}

\begin{proof}
We can see that
$$\langle f,  D_{\psi, \varphi, n }^{\ast} (K_w) \rangle=\langle D_{\psi, \varphi, n }  f, K_w\rangle=\psi (w) f^{(n)}  (\varphi (w))=\langle f, \overline{\psi(w)}  K_{\varphi(w)}^{(n)} \rangle$$
for any $f \in A_{\alpha}^2$. Hence the result follows.
\end{proof}

In the whole of this paper, we put $t=(\alpha +2)  (\alpha+3) \ldots (\alpha+n+1)$. Now we provide a few observations about $J$-symmetric operator $D_{\psi, \varphi, n}$ which will be used in the proof of Theorem \ref{f7}.

\begin{prop} \label{f3}
Suppose that an operator  $D_{\psi, \varphi, n}$ is $J$-symmetric on $A_{\alpha}^2$. Then the following hold.
\begin{itemize}
\item[$(i)$]
For each $0 \leq m <n $, $\psi^{(m)} (0)=0$;
\item[$(ii)$]
$\psi^{(n)} (0) \not = 0$;
\item[$(iii)$]
$\psi(w) \not =0$ for any $w \in \mathbb{D} \setminus  \{0 \}$;
\item[$(iv)$]
The map $\varphi$ is univalent.
\end{itemize}
\end{prop}

\begin{proof}
Suppose that $D_{\psi, \varphi, n}$ is $J$-symmetric. We observe  that
\begin{eqnarray} \label{m1}
JD_{\psi, \varphi, n} (K_{0}) =0
\end{eqnarray}
and  Lemma \ref{f2} shows that
\begin{eqnarray} \label{m2}
D_{\psi, \varphi, n}^{\ast} J(K_{0}) =\overline{\psi(0)}  K_{\varphi(0)} ^{(n)}.
\end{eqnarray}
Since  $D_{\psi, \varphi, n}$  is $J$-symmetric, by (\ref{m1}) and (\ref{m2}), we conclude that  $\psi(0) =0$. Assume that for $m < n-1$,
$\psi^{(m)}(0) =0$.  One can see that
\begin{eqnarray} \label{m3}
JD_{\psi, \varphi, n}  K_{0}^{(m+1)} =0.
\end{eqnarray}
On the other hand, we obtain
\begin{eqnarray} \label{m4}
\langle f,   D_{\psi, \varphi, n}^{\ast} J  K_0^{(m+1)} \rangle &=&  \langle f, D_{\psi, \varphi, n}^{\ast}  K_0^{(m+1)} \rangle \nonumber  \\
 &=& \langle  D_{\psi, \varphi, n} f , K_0^{(m+1)} \rangle  \nonumber \\
&=&  (\psi. (f^{(n)}  \circ  \varphi )) ^{(m+1)} (0) \nonumber  \\
&=& \sum_{i=0}^{m+1}
\displaystyle{m+1 \choose i}
\psi^{(m+1-i)} (0)  (f^{(n)}  \circ \varphi )^{(i)} (0) \nonumber  \\
&=& \psi^{(m+1)} (0)  f^{(n)} (\varphi(0))  \nonumber  \\
&+&  \sum_{i=1}^{m+1}
\displaystyle{m+1 \choose i}
\psi^{(m+1-i)} (0) (f^{(n)}  \circ   \varphi )^{(i)} (0) \nonumber  \\
&=& \psi^{(m+1)} (0)  f^{(n)}  (\varphi(0)) \nonumber \\
&=& \langle f, \overline{\psi^{(m+1)} (0) } K_{\varphi(0)} ^{(n)}   \rangle,
\end{eqnarray}
so
\begin{eqnarray}  \label{m5}
D^{\ast}_{\psi, \varphi, n} J K_0^{(m+1)}= D_{\psi, \varphi, n}^{\ast}  K_0^{(m+1)}  = \overline{\psi^{(m+1)} (0) } K_{\varphi(0)} ^{(n)}.
\end{eqnarray}
If  $D_{\psi, \varphi, n}$ is $J$-symmetric, then (\ref{m3}) and (\ref{m5}) imply that $\psi^{(m+1)} (0)=0$.
By the same idea which was seen  in  (\ref{m4}), we have
\begin{eqnarray}  \label{m6}
D_{\psi, \varphi, n}^{\ast}  JK_0^{(n)}   =D_{\psi, \varphi, n}^{\ast}  K_0^{(n)}  = \overline{\psi^{(n)} (0) } K_{\varphi(0)} ^{(n)},
\end{eqnarray}
since $\psi^{(m)} (0)=0$ for any $m <n$. Because
\begin{eqnarray} \label{gh1}
JD_{\psi, \varphi, n} K_0^{(n)}  = t  n!  J(\psi)
\end{eqnarray}
and $\psi$ is not identically $0$, by (\ref{m6}) and (\ref{gh1}), we see that $\psi^{(n)} (0) \not =0$.  Now suppose that  $\psi(w) =0$ for some $w \in \mathbb{D}$. Lemma \ref{f2} shows that
$D_{\psi, \varphi, n}^{\ast}J(K_{\overline{w}}) =0$.
Also
\begin{eqnarray*}
JD_{\psi, \varphi, n} (K_{\overline{w}}) = \frac{t  \overline{w^n}J(\psi)}{ (1 -  \overline{w}  J(\varphi) )^{n+\alpha+2}}.
\end{eqnarray*}
Since $D_{\psi, \varphi, n}$ is $J$-symmetric and $\psi$  is not identically zero, we  observe  that $w=0$.

Now assume that $D_{\psi, \varphi, n}$ is   $J$-symmetric and there exist nonzero distinct points $w_1$ and $w_2$ in $\mathbb{D}$ with
$\varphi (w_1) = \varphi(w_2)$.
One can easily   see that the kernel of $D_{\psi, \varphi, n}$ is the set of all polynomials with degree less than $n$.
Lemma \ref{f2} implies that
\begin{eqnarray*}
D_{\psi, \varphi, n}^{\ast} J(\psi (w_2) K_{\overline{w_1}} - \psi (w_1) K_{\overline{w_2}} )&=&
D_{\psi, \varphi, n}^{\ast}( \overline{\psi (w_2)} K_{w_1} - \overline{\psi (w_1)} K_{w_2})  \nonumber \\
 &=& \overline{\psi(w_1)  \psi(w_2)}  K_{\varphi (w_1)} ^{(n)} - \overline{\psi(w_1) \psi(w_2)}  K_{\varphi (w_2)} ^{(n)}=0. \nonumber
\end{eqnarray*}
Since $D_{\psi, \varphi, n}$ is $J$-symmetric, it follows that $\psi(w_2) K_{\overline{w_1}} - \psi(w_1) K_{\overline{w_2}} $ is a polynomial of degree less  than $n$. It shows that
\begin{eqnarray*}
\psi(w_2)  \sum_{j=n}^{\infty} \frac{\Gamma(j+2+\alpha)  ({w_1} )^j  z ^j}{j !  \Gamma(\alpha +2)} - {\psi(w_1)}  \sum_{j=n}^{\infty} \frac{\Gamma(j+2+\alpha)  ({w_2} )^j  z ^j}{j !  \Gamma(\alpha +2)} =0.
\end{eqnarray*}
Then $\psi (w_2)  {w_1}^m = \psi (w_1) {w_2}^m $ for each $m \geq n$.
We observe that
\begin{eqnarray*}
\psi(w_1)   w_2^{n+1} = \psi(w_2)   w_1^{n+1}  = \psi(w_2)  w_1^{n}  w_1 =  \psi(w_1)   w_2^{n} w_1,
\end{eqnarray*}
so $w_1=w_2$. If  either $w_1$ or $w_2$ is zero, by the open mapping theorem, we can find a pair of distinct points  $w_3$ and $w_4$, both nonzero with  $\varphi(w_3) = \varphi(w_4)$. Therefore $\varphi$ must be univalent.
\end{proof}

\begin{rem} \label{f666}
We can follow the outline of the proof of Proposition \ref{f3} to see that an  analogue of Proposition \ref{f3} holds for any normal operators $D_{\psi, \varphi,n}$.
\end{rem}

Suppose that $\varphi(z)= \frac{az+b}{cz+d}$ is a nonconstant linear fractional self-map of $\mathbb{D}$. Then the map $\sigma(z)=\frac{\overline{a}z - \overline{c}}{- \overline{b}z + \overline{d}}$ also takes $\mathbb{D}$ into itself (see \cite[Lemma 1]{S.1}).
Recall that if $\|\varphi \|_{\infty} <1$, then $\|\sigma\|_{\infty}<1$, and so $D_{\varphi,n}$ and $D_{\sigma,n}$ are bounded operators on $A_{\alpha}^2$.
Cowen \cite{S.1} found the adjoint of $C_{\varphi}$ acting on the Hardy space $H^2$. After that the adjoint of some weighted composition-differentiation operators $D_{\psi, \varphi, 1}$  on $H^2$ were investigated by the second and third authors
(see  \cite[Theorem 1]{s1}).
In the next result, we show that an analogue of \cite[Theorem 1]{s1} holds in the weighted Bergman spaces $A_{\alpha}^2$.\\ \par

\begin{prop}  \label{f4}
Suppose that $\varphi$ and $\sigma$   are the  linear fractional self-maps of $\mathbb{D}$ as described above. Then
\begin{eqnarray*}
D_{K^{(n)}_{\sigma(0),\varphi, n}} ^{\ast} =D_{{K^{(n)}_{\varphi(0), \sigma, n}}}.
\end{eqnarray*}
\end{prop}

\begin{proof}
We Know that
\begin{eqnarray*}
K_{\varphi(0)}^{(n)}(z) = \frac{t  z^n}{ ( 1 -  \overline{(b/d)}  z )^{n + \alpha +2}}= \frac{t   \overline{d^{n + \alpha +2}  } z^n}{  ( \overline{d}- \overline{b}z  )^{n + \alpha +2}}
\end{eqnarray*}
and
\begin{eqnarray*}
K_{\sigma(0)}^{(n)}(z) = \frac{t  z^n}{ ( 1 + (c/d)  z  )^{n + \alpha +2}}= \frac{t d^{n + \alpha +2}  z^n}{ (c  z +d  )^{n + \alpha +2}}.
\end{eqnarray*}
We see that
\begin{eqnarray}\label{a1}
D_{K^{(n)}_{\varphi(0), \sigma,n}} (K_{w})(z) &=&  T_{K_{\varphi(0)}^{(n)}}  \bigg(\frac{t  \overline{w^n}}{ (1- \overline{w}  \sigma  (z) )^{n + \alpha +2}} \bigg) \nonumber \\
&=& \frac{t^2 \overline{d^{n + \alpha +2}w^n}  z^n }{ (-\overline{b}  z + \overline{d} - \overline{w  a}  z + \overline{w  c}  )^{n + \alpha +2}}.
\end{eqnarray}
On the other hand, by Lemma \ref{f2}, we obtain
\begin{eqnarray}  \label{a2}
D_{K^{(n)}_{\sigma(0), \varphi,n}} ^{\ast} (K_{w})(z) &=&\frac{\overline{t  d^{n+\alpha +2} w^n}}{(\overline{c w}+ \overline{d} )^{n+\alpha +2}}   K_{\varphi(w)}^{(n)} (z) \nonumber \\
&=& \frac{t^2   \overline{d^{n+\alpha+2} w^n}  z^n}{ (\overline{cw} + \overline{d}- (\overline{aw} + \overline{b}) z )^{n+\alpha+2}}.
\end{eqnarray}
Since the span of the reproducing kernel functions is dense in $A_{\alpha}^2$, by  (\ref{a1}) and (\ref{a2}), the result follows.
\end{proof}

  In the next theorem, we completely describe  $J$-symmetric operators $D_{\psi, \varphi, n }$.

\begin{thm} \label{f7}
A bounded operator $D_{\psi, \varphi, n }$ is $J$-symmetric on $A_{\alpha}^2$  if and only if
\begin{eqnarray*}
\psi(z) = \frac{a}{t n!}    K_{\overline{c}}^{(n)} (z)= \frac{a  z^n}{n!   (1 - c  z)^{n + \alpha +2}}
\end{eqnarray*}
and
\begin{eqnarray*}
\varphi(z)= c + \frac{b   z}{1- c   z},
\end{eqnarray*}
where
$a=\psi^{(n)} (0)$ and  $b= \varphi^{\prime}(0)$ are both nonzero complex number and  $c=\varphi(0)$ belongs to $\mathbb{D}$.
\end{thm}

\begin{proof}
Suppose that $D_{\psi, \varphi, n} $ is $J$-symmetric.
By (\ref{m6}), (\ref{gh1}) and Proposition \ref{f3},
we conclude that  $J(\psi) = \frac{\overline{\psi^{(n)} (0)}}{t   n!}  K_{\varphi(0)}^{(n)} $ and so
$\psi= \frac{\psi^{(n)} (0)}{t   n!}  K_{\overline{\varphi(0)}} ^{(n)}$, where $\psi^{(n)}(0) \not =0 $.
 It  follows that
\begin{eqnarray} \label{m23}
\psi^{(n+1)} (0) = (n+1)   (n+\alpha+2)   \varphi(0)   \psi^{(n)} (0).
\end{eqnarray}
We have
\begin{eqnarray} \label{m24}
JD_{\psi, \varphi, n}  ( K_0^{(n+1)}) (z) &=& t   (n+1) !  (\alpha +n +2) J(\psi) (z)  J (\varphi) (z) \nonumber \\
&=& \frac{t   (n+1)  (n+\alpha+2)   \overline{\psi^{(n)}(0)}   z^n}{ (1- \overline{\varphi(0)}  z  )^{n+\alpha +2}}   J(\varphi) (z).
\end{eqnarray}
Moreover by  Proposition \ref{f3}(i), (\ref{m23}) and the proof of (\ref{m4}), we observe that
\begin{eqnarray} \label{m25}
 D_{\psi, \varphi, n}^{\ast}     J   ( K_0^{(n+1)})(z) &=&  D_{\psi, \varphi, n}^{\ast} ( K_0^{(n+1)})(z)  \nonumber \\
&=& \overline{\psi^{(n+1)}(0)}  K_{\varphi(0)} ^{(n)} (z)+  (n+1)   \overline{\psi^{(n)} (0)   \varphi^{\prime}(0)}   K_{\varphi(0)} ^{(n+1)} (z) \nonumber \\
&=& \frac{t   (n+1)  (n+ \alpha+2)   \overline{\varphi(0)  \psi^{(n)} (0)} z^n}{ (1- \overline{\varphi(0)} z )^{n+ \alpha+2}} \nonumber \\
&+& \frac{t   (n+1)   (n+ \alpha+2)   \overline{ \psi^{(n)} (0)   \varphi^{\prime}(0)  } z^{n+1}}{(1- \overline{\varphi(0)} z )^{n+ \alpha+3}}.
\end{eqnarray}
Because $ D_{\psi, \varphi, n}$  is $J$-symmetric, it follows from (\ref{m24}) and (\ref{m25}) that
\begin{eqnarray*}
J(\varphi)(z) = \overline{\varphi(0)} + \frac{\overline{\varphi^{\prime}(0)}  z }{1- \overline{\varphi(0)}  z}
\end{eqnarray*}
and so
\begin{eqnarray*}
\varphi  (z) =\varphi(0) + \frac{\varphi^{\prime}(0)   z }{1- \varphi(0)   z},
\end{eqnarray*}
with $\varphi^{\prime}(0) \not =0$ because $\varphi$ is nonconstant.

Conversely,  take $\psi$ and $\varphi$ as in the statement of  the theorem. For each $f \in A_{\alpha}^2$, we have
\begin{eqnarray} \label{f777}
JD_{\psi, \varphi,n} (f) (z) = J(\psi)(z) J (f^{(n)} (\varphi(z))) = J(\psi)(z) \overline{f^{(n)} (\varphi( \overline{z}))}.
\end{eqnarray}
On the other hand, by Proposition \ref{f4}, we see that
\begin{eqnarray} \nonumber
 D_{\psi, \varphi,n}^{\ast} J =  \frac{ \overline{a}}{{n! t}}D_{K_{\sigma(0)}^{(n)},\varphi,n}J=\frac{ \overline{a}}{{n! t}}D_{K_{\varphi(0)}^{(n)},\sigma,n}J
\end{eqnarray}
Then
\begin{eqnarray} \label{f888}
D_{\psi, \varphi,n}^{\ast} J (f)(z) = \frac{ \overline{a}}{{n! t}} K_{\varphi(0)}^{(n)}(z) \overline{f^{(n)}(\overline{\sigma(z)})} = J(\psi)(z) \overline{f^{(n)} (\varphi( \overline{z}))}.
\end{eqnarray}
Therefore by (\ref{f777}) and (\ref{f888}), the operator $D_{\psi, \varphi,n}$ is $J$-symmetric.
\end{proof}

We infer from  the paragraph after Corollary \ref{f1}, \cite[Lemma 4.8]{s5} and  the proof of \cite[Theorem 4.10]{s5}
that an operator $ D_{\psi, \varphi, n}$ from Proposition  \ref{f4} is bounded on
$A_{\alpha}^2$   if $2 |c+ \overline{c} (b-c^2)|< 1- |b-c^2|^2$.

By a similar idea as stated in the proof of  \cite[Proposition 2.1]{s0} (see also \cite[Theorem 4.1]{s55}),  we remark that $C_{\psi, \varphi}$ is unitary and $J$-symmetric on $A_{\alpha}^2$ if and only if either
\begin{eqnarray} \label{a44}
\psi(z)= \frac{\alpha (1- |p|^2)^{\frac{\alpha+2}{2}}}{(1- \overline{p}z)^{\alpha+2}}
\end{eqnarray}
and
\begin{eqnarray} \label{a444}
\varphi(z)= \frac{\overline{p}}{p} \frac{p-z}{1-\overline{p}z},
\end{eqnarray}
where $p \in \mathbb{D} \setminus \{0\}$ and $|\alpha|=1$  or
$\psi \equiv \mu$ and $\varphi(z)= \lambda z$, where $|\mu |=|\lambda |=1$.
In the case that $p \not =0$, we denote the linear functional transformations in
(\ref{a44}) and (\ref{a444}) by $\psi_p$ and
$\varphi_p$, respectively.  Invoking \cite[Lemma 2.2]{s0}, we observe that $C_{\lambda z} J$ and $C_{\psi_p, \varphi_p}J$ are conjugations. Next, we characterize complex symmetric operators $D_{\psi, \varphi, n}$ with conjugations $C_{\lambda z} J$ and $C_{\psi_p, \varphi_p}J$.\\ \par

\begin{thm} \label{f9}
Suppose  that $\tilde{\varphi}(z)= c + \frac{bz}{1-cz}$
and  that
$\tilde{\psi}(z)= \frac{az^n}{n!(1-cz)^{n+\alpha+2}}$, where $a,b \in \Bbb{C} \setminus \{0\}$ and $c \in \Bbb{D}$. Assume that $D_{\tilde{\psi}, \tilde{\varphi}, n}$ is bounded on $A_{\alpha}^2$.
\begin{itemize}
\item[$(1)$]
For $p \not =0$, an operator $D_{\psi, \varphi,n}$ on $A_{\alpha}^2$ is complex symmetric with conjugation
$C_{\psi_p, \varphi_p}J$  if and only if
$\varphi= \tilde{\varphi} \circ \varphi_p$ and $\psi= \psi_p. (\tilde{\psi} \circ  \varphi_p)$ for some $\tilde{\varphi}$ and $\tilde{\psi}$.
\item[$(2)$]
For $|\mu|=|\lambda|=1$, an operator $D_{\psi, \varphi,n}$ on $A_{\alpha}^2$
 is complex symmetric with conjugation $C_{\mu, \lambda z} J$ if and only if  $\psi(z)=\mu  \tilde{\psi}(\lambda z)$ and $\varphi(z)= \tilde{\varphi}(\lambda z)$  for some $\tilde{\varphi}$ and $\tilde{\psi}$.
\end{itemize}
\end{thm}

\begin{proof}
(1) Let $p\neq 0$. Suppose that  $D_{\psi, \varphi,n}$  is $C_{\psi_p, \varphi_p}J$-symmetric. As we mentioned in the paragraph before the statement of Theorem \ref{f9}, the operator   $C_{\psi_p, \varphi_p}^{\ast}$
is unitary and $J$-symmetric, so it is not hard to see that $C_{\psi_p, \varphi_p}^{\ast}$ is $C_{\psi_p, \varphi_p}J$-symmetric.
 Then \cite [Proposition 2.3]{s0} implies that $C_{\psi_p, \varphi_p}^{\ast} D_{\psi, \varphi,n}$
is $J$-symmetric. It results  from Theorem \ref{f7} that there is a $J$-symmetric operator $D_{\tilde{\psi}, \tilde{\varphi},n}$ so that $D_{\psi, \varphi, n}= C_{\psi_p, \varphi_p} D_{\tilde{\psi}, \tilde{\varphi},n}$.
 Hence we observe that
$\varphi= \tilde{\varphi} \circ \varphi_p$ and $\psi= \psi_p. (\tilde{\psi} \circ \varphi_p)$.

Conversely, suppose that $\varphi= \tilde{\varphi} \circ \varphi_p$ and $\psi= \psi_p. (\tilde{\psi} \circ \varphi_p)$ for some $\tilde{\varphi}$ and $\tilde{\psi}$.
Then $D_{\psi, \varphi,n}= C_{\psi_p, \varphi_p} D_{\tilde{\psi}, \tilde{\varphi},n}$. Since  the weighted composition operator
$ C_{\psi_p, \varphi_p} $ is unitary and $J$-symmetric and  the operator  $D_{\tilde{\psi}, \tilde{\varphi},n}$ is $J$-symmetric too (see Theorem \ref{f7}), the operator $D_{\psi, \varphi, n}$ is
$ C_{\psi_p, \varphi_p}J$-symmetric by \cite[Proposition 2.3]{s0}.

$(2)$ The result follows immediately from the technique as  stated in the proof of Part $\hbox{(1)}$.
\end{proof}

\section{Some examples of complex symmetric operators}

   In this section, we see that the class of $J$-symmetric and $C_{\lambda z}J$-symmetric  $D_{\psi, \varphi, n}$ contain self-adjoint $D_{\psi, \varphi, n}$ and some normal operators $D_{\psi, \varphi, n}$.
   In the next proposition, we obtain a characterization of   self-adjoint weighted composition-differentiation operators of order $n$  on $A^{2}_{\alpha}$.

\begin{prop} \label{f6}
A bounded operator $D_{\psi, \varphi, n}$ is self-adjoint  on $A_{\alpha}^{2}$  if and only if
\begin{eqnarray*}
\psi(z) = \frac{a   z^n}{n! (1 - \overline{c}   z)^{n + \alpha+2}}= \frac{a}{t n!} K_c^{(n)} (z)
\end{eqnarray*}
and
\begin{eqnarray*}
\varphi(z) = c+ \frac{b  z}{1 - \overline{c}   z},
\end{eqnarray*}
where $a= \psi^{(n)} (0)$ and $b= \varphi^{\prime}(0)$ are both nonzero real numbers  and $c= \varphi(0)$ belongs to $\mathbb{D}$.
 Furthermore, for the self-adjoint operator  $D_{\psi, \varphi,n}$ either of the following holds:
\begin{itemize}
\item[$i)$] If $c=0$, then $D_{\psi, \varphi,n}$ is $J$-symmetric.
\item[$ii)$]  If $c \not =0$, then $D_{\psi, \varphi,n}$ is $C_{e^{-2 i \theta}z}J$-symmetric, where $\theta=\mbox{Arg}(c)$.
\end{itemize}
\end{prop}

\begin{proof}
Suppose that  $D_{\psi, \varphi, n}$ is self-adjoint on $A_{\alpha}^2$.  By  (\ref{m4}) and Remark \ref{f666},  we have
$D_{\psi, \varphi, n }^{\ast} K_0 ^{(n)} = \overline{\psi^{(n)}(0)}   K_{\varphi(0)}^{(n)}$.
Moreover, we can see that
$D_{\psi, \varphi, n }  K_0 ^{(n)}(z) = D_{\psi, \varphi, n } (t z^n ) = t n! \psi (z)$.
Since $D_{\psi, \varphi, n}$ is self-adjoint, we conclude that
\begin{eqnarray} \label{m16}
\psi(z) = \frac{\overline{\psi^{(n)} (0)}}{t  n!}   K_{\varphi(0)}^{(n)}(z)= \frac{\overline{\psi^{(n)} (0)} z^n}{n!   ( 1 - \overline{\varphi(0)} z )^{n+\alpha+2}}.
\end{eqnarray}
Differentiating  both sides of (\ref{m16}) $n$ times with respect to $z$,  we obtain
\begin{eqnarray} \label{m17}
\psi^{(n)} (z) =\frac{ \overline{\psi^{(n)} (0)}}{n!} \sum_{i=0}^n
\displaystyle{n \choose i}
 \frac{n!}{i!}  z^i  \bigg( \frac{1}{(1- \overline{\varphi(0)} z)^{n+\alpha +2}} \bigg)^{(i)}.
\end{eqnarray}
It results from  (\ref{m17})  that
$\psi^{(n)} (0) = \overline{\psi^{(n)} (0)}$ and so $\psi^{(n)} (0)$ is real. Moreover, note that $\psi^{(n)}(0)\neq 0$ since $\psi$ is not identically $0$.
On the other hand, differentiating the left side and the right side of (\ref{m16}) $n+1$ times with respect to $z$ yields
\begin{eqnarray} \label{a3}
\psi^{(n+1)} (0) = (n+1)   (n+\alpha +2)    \overline{\varphi(0)}   \psi^{(n)} (0).
\end{eqnarray}
We can see that
\begin{eqnarray} \label{m18}
D_{\psi, \varphi, n} (K_0^{(n+1)} ) (z)  &=& D_{\psi, \varphi, n}  (t (n+ \alpha+2)   z^{n+1} )  \nonumber \\
&=&  \frac{t   (n+1)   (n+ \alpha+2)   \psi^{(n)}(0)   z^n}{(1- \overline{\varphi(0)}   z )^{n+ \alpha+2}}    \varphi(z).
\end{eqnarray}
On the other hand, by the idea  as stated in (\ref{f3}) and the fact that for each $m<n$, $\psi^{(m)}(0) =0$ (see Remark \ref{f666}),  we have
\begin{eqnarray} \label{m19}
D_{\psi, \varphi, n}^{\ast} (K_0^{(n+1)} ) (z) &=&  \overline{\psi^{(n+1)}(0)}   K_{\varphi(0)}^{(n)} (z) + (n+1)  \overline{\psi^{(n)}(0) \varphi^{\prime}(0)}    K_{\varphi(0)}^{(n+1)}(z)  \nonumber \\
&=& \frac{t  \overline{\psi^{(n+1)} (0)}  z^n}{(1- \overline{\varphi(0)} z )^{n + \alpha +2}} \nonumber \\
&+& \frac{(n+1) \overline{ \psi^{(n)}(0)  \varphi^{\prime}(0)} t ( n+ \alpha+2)   z^{n+1} }{(1- \overline{\varphi(0)} z  )^{n + \alpha +3}}.
\end{eqnarray}
Since $D_{\psi, \varphi, n}$ is self-adjoint, by calling   (\ref{a3}), (\ref{m18}) and (\ref{m19}), we get
\begin{eqnarray} \label{m20}
\varphi(z) = \varphi(0) + \frac{\overline{\varphi^{\prime}(0)}   z}{1- \overline{\varphi(0)} z }.
\end{eqnarray}
Differentiating both sides of  (\ref{m20})  with respect to $z$ and  then taking $z=0$, we observe that $\varphi^{\prime}(0)$ is also real.
In addition, because $\varphi$ is not constant, we see that $\varphi^{\prime}(0)  \not =0 $.

For the converse, suppose that $\varphi$ and $\psi$ are as in the statement of the  proposition and $C_{\psi, \varphi,n}$ is bounded on $A_{\alpha}^2$.  Proposition \ref{f4} dictates that
\begin{eqnarray*}
D_{\psi, \varphi, n}^{\ast}= \frac{\overline{a}}{tn!}D^{\ast}_{K_{\sigma(0)}^{(n)},\varphi,n}=\frac{\overline{a}}{tn!}
D_{K_{\varphi(0)}^{(n)},\sigma,n}=D_{\psi, \varphi, n}.
\end{eqnarray*}
Then $D_{\psi, \varphi, n}$ is self-adjoint.

We infer from Theorem \ref{f7} that for  the case $c=0$,  the operator $D_{\psi, \varphi,n}$ is $J$-symmetric.
Now let $c \not =0$.  Set  $\tilde{\psi}(z)= \frac{a e^{2n i \theta} z^n}{n! (1- cz)^{n+\alpha+2}}$
and  $\tilde{\varphi}(z)=c+ \frac{b e^{2 i \theta} z}{1- cz}$.
From Theorem \ref{f7}, the operator  $ D_{\tilde{\psi}, \tilde{\varphi},n}$ is  $J$-symmetric. By
\cite[Lemma 2.2]{s0} and \cite[Proposition 2.3]{s0}, we observe that  $C_{e^{-2 i \theta}z} D_{\tilde{\psi}, \tilde{\varphi},n}$
is  $C_{e^{-2 i \theta}z}J$-symmetric.
(note that as stated in the paragraph before Theorem \ref{f9}, the composition operator   $C_{e^{-2 i \theta}z}$ is unitary and  $J$-symmetric.) A direct computation shows that
$ C_{e^{-2 i \theta}z} D_{\tilde{\psi}, \tilde{\varphi},n}= D_{\psi, \varphi,n}$, so the result follows.
\end{proof}

 Now we characterize those   operators  $D_{\psi, \varphi, n}$ on $A_{\alpha}^2$ that are normal when $0$ is the fixed point of $\varphi$.

\begin{prop} \label{f5}
Suppose that  an operator $D_{\psi, \varphi, n}$ is bounded on  $A_{\alpha}^2$ and that $\varphi(0)=0$. Then $D_{\psi, \varphi, n}$ is normal if and only if $\psi(z) = a  z^n$ and $\varphi(z)= b  z$, where $a \in \Bbb{C}  \setminus \{0\}$ and $b$ belongs to
$\mathbb{D} \setminus  \{0 \}$.
Moreover, in this case   $D_{\psi, \varphi, n}$ is $J$-symmetric.
\end{prop}

\begin{proof}
Assume that  $D_{\psi, \varphi, n}$ is normal on  $A_{\alpha}^2$. We can see that
\begin{eqnarray} \label{m8}
\|D_{\psi, \varphi, n}   K_0^{(n)} \|^2 =  \bigg \|  \left(\frac{n!}{\beta(n)}\right)^{2}  \psi  \bigg \|^2
= \bigg( \frac{n!}{\beta(n)}\bigg )^4  \sum_{j=0}^{\infty} \bigg(\frac{\beta(j)}{j !}  \bigg)^2  |\psi^{(j)} (0) |^2.
\end{eqnarray}
On the other hand,  by (\ref{m4}) and Remark \ref{f666}, we observe that
\begin{eqnarray} \label{m9}
\|D_{\psi, \varphi, n} ^{\ast}  K_0^{(n)} \|^2  = \| \overline{\psi^{(n)} (0)} K_0 ^{(n)} \|^2
=  |\psi^{(n)} (0) |^2  \bigg( \frac{n!}{\beta(n)} \bigg)^2.
\end{eqnarray}
Because $D_{\psi, \varphi, n}$ is normal, by Remark  \ref{f666}, (\ref{m8}) and (\ref{m9}),  we conclude that
\begin{eqnarray} \label{m10}
|\psi^{(n)} (0) |^2   \bigg( \frac{n!}{\beta(n)} \bigg)^2 = \bigg( \frac{n!}{\beta(n)} \bigg)^4  \sum_{j=n}^{\infty}   \bigg(\frac{\beta(j)}{j!}\bigg )^2  |\psi^{(j)} (0) |^2.
\end{eqnarray}
Remark \ref{f666} implies that $\psi^{(n)}(0)\neq 0$, so  from (\ref{m10}),  for each $j>n$, $\psi^{(j)} (0)=0$.
Since Remark \ref{f666} also  shows that for any $j<n$, $\psi^{(j)} (0)=0$, the map $\psi$ must be of the form
$\psi(z)= a  z^n$,
 for some $a \in \Bbb{C} \setminus \{0 \}$. We have
\begin{eqnarray} \label{m11}
D_{\psi, \varphi, n} (K_0^{(n+1)})(z) &=& \bigg( \frac{(n+1)!}{\beta  (n+1)} \bigg)^2  \psi(z)  \varphi(z)  \\ \nonumber
&=&\bigg( \frac{(n+1)!}{\beta  (n+1)} \bigg)^2  a   z^n \varphi (z).
\end{eqnarray}
On the other hand, by using (\ref{m4}) and the fact that for each $m \not = n$, $\psi^{(m)} (0)=0$, we observe that
\begin{eqnarray} \label{m12}
D_{\psi, \varphi, n}^{\ast} (K_0^{(n+1)}) (z)  &=&  (n+1)  \overline{\psi^{(n)} (0)  \varphi^{\prime} (0)}   K_{0}^{(n+1)}(z)  \nonumber \\
&=&  \overline{a  \varphi^{\prime} (0)}  \bigg( \frac{(n+1)!}{ \beta  (n+1)}\bigg)^2   z^{n+1}. \nonumber \\
&=& \overline{a \varphi^{\prime}(0)} (n+1)! K_0^{(n+1)}(z),
\end{eqnarray}
so $K_0^{(n+1)}$ is an eigenvalue for $D_{\psi, \varphi, n}^{\ast}$  corresponding to eigenvalue
$\overline{a \varphi^{\prime}(0)} (n+1)!$. Therefore
\begin{eqnarray} \label{mahsa10}
D_{\psi, \varphi,n} K_0^{(n+1)} =  a \varphi^{\prime}(0) (n+1)! K_0^{(n+1)}.
\end{eqnarray}
Since $D_{\psi, \varphi, n}$ is normal on $A_{\alpha}^2$, by (\ref{m11}) and (\ref{mahsa10}), we see that
\begin{eqnarray*}
a \varphi^{\prime}(0) (n+1)! K_0^{(n+1)}(z) = \bigg ( \frac{(n+1)!}{\beta(n+1)} \bigg)^2  a  z^n  \varphi (z).
\end{eqnarray*}
Then $\varphi (z) = \varphi^{\prime}(0) z$.  Because $\varphi$ is not identically $0$, we conclude that $\varphi(z)=bz$for some $b \in \mathbb{D} \setminus \{0\}$.

For  the converse, take $\psi$ and $\varphi$ as in the statement of the proposition and assume that $D_{\psi, \varphi,n}$
is bounded on $A_{\alpha}^2$. Proposition \ref{f4} implies that
$D^{\ast}_{az^n, bz,n}= D_{\overline{a} z^n, \overline{b}z,n}$. Then for each  $f \in A_{\alpha}^2$, after some computation, we have
\begin{eqnarray} \label{f999}
D_{az^n, bz,n} D^{\ast}_{az^n, bz,n}(f)(z) &=& D_{az^n, bz,n}  D_{\overline{a} z^n, \overline{b}z,n} (f)(z) \nonumber \\
&=& D_{az^n, bz,n} (\overline{a} z^n f^{(n)}  (\overline{b}z)) \nonumber \\
&=& |a|^2  z^n \sum_{i=0}^{n}
\displaystyle{n \choose i}
\frac{n!}{i!} |b|^{2i} z^i f^{(n+i)} (|b|^2 z);
\end{eqnarray}
similarly
\begin{eqnarray} \label{f9990}
D^{\ast}_{az^n,bz,n} D_{az^n, bz,n}(f)(z)= |a|^2  z^n \sum_{i=0}^{n}
\displaystyle{n \choose i}
\frac{n!}{i!} |b|^{2i} z^i f^{(n+i)} (|b|^2 z).
 \end{eqnarray}
Then (\ref{f999}) and (\ref{f9990}) show that $D_{\psi, \varphi,n}$ is normal. Furthermore, Theorem \ref{f7} shows that
$D_{\psi, \varphi,n}$ is $J$-symmetric.
\end{proof}

Here we describe for which constant $a,b,c$, the analytic functions  $\varphi$ and $\psi$  that were obtained in Proposition \ref{f6} induce normal operator $D_{\psi, \varphi,n}$.

\begin{prop} \label{f8}
Suppose that $D_{\psi, \varphi,n}$ is a  bounded operator, with
\begin{eqnarray*}
\psi(z)= \frac{az^n}{n!(1-\overline{c}z)^{n+\alpha+2}}
\end{eqnarray*}
and
\begin{eqnarray*}
\varphi(z)=c + \frac{bz}{1-\overline{c}z},
\end{eqnarray*}
where $a=\psi^{(n)} (0)$ and $b=\varphi^{\prime}(0)$ are both nonzero complex numbers and $c= \varphi(0)$ belongs to
$\mathbb{D}$.
The operator $D_{\psi, \varphi,n}$ is normal on $A_{\alpha}^2$ if and only if either $b$ belongs to $\Bbb{R} \setminus \{0\}$
 or $c=0$. Moreover,   in this case of normal operator $D_{\psi, \varphi,n}$, either of the following holds:
\begin{itemize}
\item[$i)$] If $c=0$, then $D_{\psi, \varphi,n}$ is $J$-symmetric.
\item[$ii)$]  If $c \not =0$, then $D_{\psi, \varphi,n}$ is $C_{e^{-2 i \theta}z}$J-symmetric, where $\theta=\mbox{Arg}(c)$.
\end{itemize}
\end{prop}

\begin{proof}
Suppose that $b\in \Bbb{R} \setminus \{0\}$ or $c=0$.
Propositions \ref{f6} and \ref{f5} imply that  $D_{\psi, \varphi,n}$ is normal.

For the converse, suppose that $b$ and $c$ belong to $\Bbb{C} \setminus \Bbb{R}$. We have
$$D_{\psi, \varphi,n}(K_{\frac{1}{2}}) (z)=\frac{t \psi(z)}{2^n (1 - \frac{1}{2} \varphi(z))^{n+\alpha+2}}=\frac{a}{2^n n! (1- c/2 )^{n + \alpha+2}} K_{p_1}^{(n)} (z),$$
where $p_1=c + \frac{\overline{b}/2}{1-\overline{c}/2}$.
On the other hand,  by Lemma \ref{f2}, we see that
$$D_{\psi, \varphi,n}^{\ast}(K_{\frac{1}{2}}) (z)=\overline{\psi(1/2)} K_{\varphi(1/2)} ^{(n)}(z)=\frac{\overline{a}}{2^n n! (1-c/2)^{n+\alpha+2}} K_{p_2}^{(n)}(z),$$
where   $p_2=c+ \frac{b/2}{1-\overline{c}/2}$.

If  $D_{\psi, \varphi,n}$ were normal, then
\begin{eqnarray*}
\|D_{\psi, \varphi,n} (K_{\frac{1}{2}}) \|^2 &=& \bigg | \frac{a}{2^n n! (1- c/2)^{n + \alpha+2}}  \bigg|^2  \|  K_{p_1}^{(n)}\|^2 \nonumber \\
&=&  \bigg | \frac{a}{2^n n! (1- c/2)^{n + \alpha+2}}  \bigg|^2   \sum_{j=n}^{\infty} \frac{(|p_1|^2)^{j-n}}{\beta(j)^2} \bigg(\frac{j!}{(j-n)!} \bigg)^2 \nonumber
\end{eqnarray*}
would equal
\begin{eqnarray*}
\|D_{\psi, \varphi,n}^{\ast} (K_{\frac{1}{2}}) \|^2 &=& \bigg | \frac{a}{2^n n! (1-c/2 )^{n + \alpha+2}}  \bigg|^2  \|  K_{p_2}^{(n)}\|^2 \nonumber \\
&=&  \bigg | \frac{a}{2^n n! (1-c/2 )^{n + \alpha+2}}  \bigg|^2   \sum_{j=n}^{\infty} \frac{(|p_2|^2)^{j-n}}{\beta(j)^2} \bigg(\frac{j!}{(j-n)!} \bigg)^2. \nonumber
\end{eqnarray*}
Therefore  $|p_1|^2=|p_2|^2$ and so $c=\overline{c}$ which is a contradiction.
Now if $D_{\psi, \varphi,n}$ were normal and $b \in \Bbb{C} \setminus \Bbb{R}$ and $c \in \Bbb{R} \setminus \{0\}$,
then by  the similar idea as stated, we can see that
 $\| D_{\psi, \varphi,n}^{\ast}K_{\frac{i}{2}}\| \not = \| D_{\psi, \varphi,n} K_{\frac{i}{2}} \| $
which is a contradiction.
The rest of the proof is obtained by  the similar argument as stated in Proposition   \ref{f6}.
\end{proof}

 \end{document}